\documentclass[12pt]{amsart}
\usepackage{amsfonts,color,morefloats,pslatex}
\usepackage{amssymb,amsthm, amsmath,latexsym}

\newtheorem{theorem}{Theorem}
\newtheorem{lemma}[theorem]{Lemma}

\newtheorem{example}{Example}

\def\lcm{\mathrm{lcm}}
\def\tr{\mathrm{Tr}}

\def\qi#1 {\fbox {\footnote {\ }}\ \footnotetext { From Qi: {\color{red}#1}}}

\begin{document}
\title[Infinite families of $2$-designs from binary cyclic codes]{Infinite families of $2$-designs from two classes of binary cyclic codes with three nonzeros}

\author{Xiaoni Du}
\address{College of Mathematics and Statistics, Northwest Normal University, Lanzhou, Gansu 730070, China}
\curraddr{}
\email{ymldxn@126.com}
\author{Rong Wang}
\address{College of Mathematics and Statistics, Northwest Normal University, Lanzhou, Gansu 730070, China}
\curraddr{}
\email{rongw113@126.com}
\author{Chunming Tang}
\address{School of Mathematics and Information, China West Normal University, Nanchong, Sichuan 637002, China}
\curraddr{}
\email{tangchunmingmath@163.com}
\author{Qi Wang}
\address{Department of Computer Science and Engineering, Southern University of Science and Technology, Shenzhen, Guangdong 518055, China}
\curraddr{}
\email{wangqi@sustech.edu.cn}

\thanks{}
\date{\today}

\maketitle

\begin{abstract}
Combinatorial $t$-designs have been an interesting topic in combinatorics for decades. It is a basic fact that  the codewords of a fixed weight in a code may hold a $t$-design. Till now only a small amount of work on constructing $t$-designs from codes has been done. In this paper, we determine the weight distributions of two classes of cyclic codes: one related to the triple-error correcting binary BCH codes, and the other related to the cyclic codes with parameters satisfying the
generalized Kasami case, respectively. We then obtain infinite families of $2$-designs from these codes by proving that they are both affine-invariant codes, and explicitly determine their parameters. In particular, the codes derived from the dual of binary BCH codes hold five $3$-designs when $m=4$.

\end{abstract}

\section{Introduction}
Let $\mathcal{P}$ be a set of $v \geq1$ elements and $\mathcal{B}$ be a set of $k$-subsets of $\mathcal{P},$ where $k$ is a positive integer with $1\leq k \leq v$. Let $t$ be a positive integer with $t\leq k.$ If every $t$-subset of $\mathcal{P}$ is contained in exactly $\lambda$ elements of $\mathcal{B}$, then we call the pair $\mathbb{D}=(\mathcal{P}, \mathcal{B})$ a $t$-$(v,k,\lambda)$ {\em design}, or simply a {\em $t$-design}. The elements of $\mathcal{P}$ are called
{\em points}, and those of $\mathcal{B}$ are referred to as {\em blocks}. We often denote the number of blocks by $b$ and a $t$-design is simple when there is no repeated blocks in $\mathcal{B}$. A $t$-design is called {\em symmetric} if $v=b$ and {\em trivial} if $k=t$ or $k=v$. Throughout this paper we study only simple $t$-designs with $t < k < v$. When $t\geq 2$ and $\lambda=1$, a $t$-design is called a {\em Steiner system}. Clearly, the parameters of a $t$-$(v,k,\lambda)$ design are restricted by the following identity. 
\begin{equation}\label{condition}
  b {k \choose t} = \lambda {v \choose t}.
\end{equation}

The interplay between codes and $t$-designs has been ongoing for decades. On  one hand, the incidence matrix of a $t$-design over any finite field can serve as a generator matrix of a linear code and much progress has been made (see~\cite{AK92,Ding15b,KP95,KP03,Ton98,Ton07}). On the other hand, linear and nonlinear codes may hold $t$-designs. As a classical example, $4$-designs and $5$-designs with certain parameters were derived from binary and ternary Golay codes. Recently, Ding and Li~\cite{DL17} obtained infinite families of $2$-designs from $p$-ary Hamming codes, ternary projective cyclic codes, binary codes with two zeros and their duals. They also obtained $3$-designs from the extended codes of these codes and RM codes. More recently, infinite families of $2$-designs and $3$-designs from some classes of binary linear codes with five weights were given by Ding ~\cite{Ding182}. For other constructions of $t$-designs, for example, see~\cite{BJL99,CM06,MS771,RR10}.

The objective of this paper is to construct $2$-designs from two classes of cyclic codes obtained from the triple-error correcting narrow-sense primitive BCH codes and the cyclic codes related to the generalized Kasami case, respectively. In the following, we will first present the weight distributions of these two classes of cyclic codes, and then explicitly determine the parameters of the derived $2$-designs.

\section{The classical construction of $t$-designs from affine-invariant codes }

Throughout this paper, let $p=2$, $m=2s$, $\gcd(s, l)=d$ and $\gcd(s+l, 2l)=d'$, where both $s\geq 2$ and $1\leq l\leq m-1$ are positive integers with $l\neq s$. Let
$\mathbb{F}_q$ denote the finite field with $q=2^m$ elements and
$\mathbb{F}^*_q=\mathbb{F}_q\backslash\{0\}$. An $[n,k,d]$ {\em linear code }
  $\mathcal{C}$ over $\mathbb{F}_2$ is a $k$-dimensional subspace of $\mathbb{F}_2^n$
with minimum Hamming distance $d$, and is {\em cyclic} if any cyclic shift of a codeword is another codeword of $\mathcal{C}$.
 Any
 cyclic code $\mathcal{C}$ can be expressed as $\mathcal{C} = \langle g(x) \rangle,$
 where $g(x)$ is monic and has the least degree. The polynomial  $g(x)$ is called the {\em generator polynomial}  and $h(x)=(x^n-1)/g(x)$ is referred to
 as the
{\em  parity-check polynomial}  of $\mathcal{C}$.  If the generator polynomial
$g(x)$ (resp. the parity-check polynomial $h(x)$) can be factored
into a product of $r$ irreducible polynomials over $\mathbb{F}_p$, then
$\mathcal{C}$ is called a cyclic code with {\em $r$ zeros}  (resp. {\em $r$ nonzeros}).
The code with the generator polynomial $x^kh(x^{-1})$ is called the {\em dual} of $\mathcal{C}$ and denoted by
$\mathcal{C}^{\bot}$.

Furthermore, we define the {\em extended} code  of a code $\mathcal{C}$ to be
the code
$$
\overline{\mathcal{C}}=\{(c_0, c_1, \ldots, c_n) \in \mathbb{F}_2^{n+1}:(c_0, c_1, \ldots, c_{n-1}) \in \mathcal{C} ~with ~\sum^n_{i=0}c_i=0\}.
$$ 
The
 {\em support} of a codeword $\mathbf{c}$ is defined by
$$Suppt(\mathbf{c})=\{0\leq i \leq n-1: c_i\neq 0\}.$$
Let $A_i$ be the number of codewords with Hamming weight $i$ in a code $\mathcal{C}$. The   {\em weight enumerator} of $\mathcal{C}$ is defined by
$$1+A_1z+A_2z^2+\ldots+A_nz^n.$$
The sequence $(1, A_1, \ldots, A_n)$ is called the  {\em weight distribution} of the code $\mathcal{C}.$ If $|\{1\leq i\leq n: A_i\neq 0\}|=w,$ then we
call
$\mathcal{C}$ a  {\em $w$-weight code}.

Let $n=q-1$, and $\alpha$ be a generator of $\mathbb{F}^*_q$. For any $i$ with $0\leq i
\leq n-1$, let $M_i(x)$  denote the {\em minimal polynomial} of $\alpha^i$ over $\mathbb{F}_2$. For any
$2\leq\delta\leq n$, the code $\mathcal{C}_{(p, n, \delta)}= \langle g_{(p, n, \delta, 1)} \rangle$ with
$$g_{(p, n, \delta,
1)}(x) = \lcm(M_1(x), M_{2}(x), \ldots,  M_{1+\delta-2}(x)),$$
where $\lcm$ denotes the least common multiple of the polynomials,   is called a  {\em narrow-sense primitive BCH code} with designed distance $\delta$.

For each $i$ with $A_i\neq 0$, let $\mathcal{B}_i$ denote the set of the supports of all codewords with weight $i$ in $\mathcal{C}$, where the coordinates
of a codeword are indexed by $(0, 1, 2, \ldots, n-1).$
Let $\mathcal{P}=\{0, 1, \ldots, n-1\}.$ The pair $(\mathcal{P},\mathcal{B}_i)$ could be a $t$-$(v,i,\lambda)$ design for a certain positive $\lambda$~\cite{Ton98}. There exist two classical approaches to obtain $t$-designs from linear codes. The first one is to employ the Assmus-Mattson Theorem given in~\cite{AM69}, and the second one is to  study the automorphism group of a linear code $\mathcal{C}$. If the permutation part of the automorphism group acts $t$-transitively on a code $\mathcal{C}$, then the code $\mathcal{C}$ holds $t$-designs~\cite{AK92,MS771}. In the following, we will use the latter method to construct $2$-designs.

We conclude this section by summarizing some known results on affine-invariant codes related to  $2$-designs.

The $2$-adic expansion of each $e\in\mathcal{P}$ is given by
$$
e=\sum^{m-1}_{i=0}e_i2^i,~ ~0\leq e_i\leq 1 ,~0\leq i \leq m-1.
$$
For any $r=\sum^{m-1}_{i=0}r_i2^i\in\mathcal{P},$
we say that $r\preceq e$ if $r_i \leq e_i$ for all $0\leq i\leq m-1.$ By definition, we have $r\leq e$ if $r\preceq e$.

The set of coordinate permutations that map a code $\mathcal{C}$ to itself forms a group, which is referred to as the  {\em permutation automorphism
group} of $\mathcal{C}$ and denoted by $PAut(\mathcal{C})$. We define the {\em affine group} $GA_1(\mathbb{F}_q)$ by the set of all permutations
$$\sigma_{a,b}: x \mapsto ax+b$$
of $\mathbb{F}_q$, where $a\in \mathbb{F}_q^*$ and $b \in \mathbb{F}_q$. An affine-invariant code is an extended cyclic code $\overline{\mathcal{C}}$ over
$\mathbb{F}_2$ such that $GA_1(\mathbb{F}_q)\subseteq PAut(\overline{\mathcal{C}})$~\cite{HP03}.


For any integer $0\leq j< n$, the $2$-cyclotomic coset of $j$ modulo $2^m-1$ is defined by
$$C_j=\{jp^i \pmod {2^m-1} : 0 \le i\leq \ell_j-1\},$$
where $\ell_j$ is the smallest positive integer such that $j\equiv jp^{\ell_j}\pmod {2^m-1}.$
Let $g(x)=\prod_j\prod_{i\in C_j}(x-\alpha^i)$, where $j$ runs through some subset of  representatives of  the $2$-cyclotomic cosets
$C_j$ modulo $2^m-1.$  The set   $T=\bigcup_jC_j$ is called the {\em defining set} of $\mathcal{C}$, which is the union of these $2$-cyclotomic cosets.

Affine-invariance is an important property of an extended primitive cyclic code, for which the following lemma presented by Kasami, Lin and Peterson~\cite{KLP67} provides a sufficient and necessary condition by examining the defining set of the code.

\begin{lemma}\label{Kasami-Lin-Peterson}~\cite{KLP67}
  Let $\overline{\mathcal{C}}$ be an extended cyclic code of
  length $2^m$ over $\mathbb{F}_2$ with defining set $\overline{T}$. The code $\overline{\mathcal{C}}$ is affine-invariant if and
only if whenever $e \in \overline{T}$ then $r \in \overline{T}$ for all $r \in \mathcal{P}$  with $r \preceq e$.
\end{lemma}

\begin{lemma}\label{The dual of an affine-invariant code}~\cite{Ding18b} 
  The dual of an affine-invariant code $\overline{\mathcal{C}}$ over $\mathbb{F}_2$ of length $n+1$
is also affine-invariant.
\end{lemma}

The importance of affine-invariant codes is partly due to Theorem \ref{2-design} which  will be used together with Lemmas~\ref{Kasami-Lin-Peterson} and
\ref{The dual of an affine-invariant code} to derive the existence of $2$-designs.

\begin{theorem}\label{2-design}~\cite{Ding18b}
  For each $i$ with $\overline{A_i} \neq 0$ in an affine-invariant code $\overline{\mathcal{C}}$, the supports of the codewords of weight $i$
form a $2$-design.
\end{theorem}

\section{Two classes of cyclic codes and their $t$-designs}\label{theorem}

In this section, we introduce the main results on the weight distributions of two classes of cyclic codes and the corresponding $2$-designs. Their proofs will be presented in the subsequent section. In the following, let $\tr_1^m$ denote the trace function from $\mathbb{F}_{2^m}$ onto $\mathbb{F}_2$.

\subsection{Results on the linear code derived from triple-error correcting BCH code}

We define
\begin{eqnarray}\label{code-1}
  {\overline{{\mathcal{C}_1}^{\bot}}}^{\bot}:=\{ (\tr_1^m (ax^5+bx^3+cx)+h )_{x \in \mathbb{F}_q}: a,b,c \in \mathbb{F}_q, h\in \mathbb{F}_2 \},
\end{eqnarray}
where $\mathcal{C}_1$ is the cyclic code of length $n$ with the parity-check polynomial $M_1(x)M_3(x)M_5(x)$. It is easily seen that $\mathcal{C}_1^{\bot}$ is a BCH code with minimum distance $d\geq\delta=7$. Note that for $\mathcal{C}_1^{\bot}$, we only discuss the case of $m$ even since the complement case for $m$ odd has been studied in~\cite{Ding182}.

The following two theorems constitute the first part of our main results in the present paper.

\begin{theorem}\label{weight1}
Let $s\geq 3.$ The weight distributions of the code ${\overline{{\mathcal{C}_1}^{\bot}}}^{\bot}$ over $\mathbb{F}_2$ with length $n+1$ and
$dim({\overline{{\mathcal{C}_1}^{\bot}}}^{\bot})=3m+1$ are given in Table \ref{1}.
\begin{table}
\begin{center}
\caption{The weight distribution of ${\overline{{\mathcal{C}_1}^{\bot}}}^{\bot}$}\label{1}
\begin{tabular}{ll}
\hline\noalign{\smallskip}
Weight  &  Multiplicity   \\
\noalign{\smallskip}
\hline\noalign{\smallskip}
$0$  &  1 \\
$2^{2s-1}$  &  $ 29\times2^{6s-5}-33\times2^{4s-5}+17\times2^{2s-3}-2 $    \\
$2^{2s-1}-2^{s-1} $ & $ \frac{2}{15}\times2^{2s}(3\times2^{4s}+5\times2^{2s}-8)$\\
$ 2^{2s-1}+2^{s-1}$  &  $ \frac{2}{15}\times2^{2s}(3\times2^{4s}+5\times2^{2s}-8)$\\
$2^{2s-1}-2^s$  &  $ \frac{7}{3}\times2^{4s-4}(2^{2s}-1)$     \\
$ 2^{2s-1}+2^s$  &  $ \frac{7}{3}\times2^{4s-4}(2^{2s}-1)$     \\
$ 2^{2s-1}-2^{s+1} $  &  $\frac{1}{15}\times2^{2s-4}(2^{4s-2}-5\times2^{2s-2}+1)$     \\
$ 2^{2s-1}+2^{s+1}$  &  $\frac{1}{15}\times2^{2s-4}(2^{4s-2}-5\times2^{2s-2}+1)$     \\
$ 2^{2s}$  &  $ 1$     \\
\noalign{\smallskip}
\hline
\end{tabular}
\end{center}
\end{table}
\end{theorem}

Note that the codes defined in (\ref{code-1}) are eight-weight.

\begin{theorem}\label{$2-$design-1}
Let $s\geq 3$ be a positive integer. Then the supports of the codewords of weight $i$ with ${\overline{{A_i}^{\bot}}}^{\bot}\neq0$ in
${\overline{{\mathcal{C}_1}^{\bot}}}^{\bot}$ form a $2$-design.
Moreover, let $\mathcal{P}=\{0, 1, \ldots, 2^m-1\}$ and $\mathcal{B}$ be the set of the supports of the codewords of
${\overline{{\mathcal{C}_1}^{\bot}}}^{\bot}$ with weight $i,$ where
${\overline{{A_i}^{\bot}}}^{\bot}\neq 0.$ Then ${\overline{{\mathcal{C}_1}^{\bot}}}^{\bot}$ holds $2$-$(2^m, i, \lambda)$ designs for the following
pairs:
\begin{itemize}
\item $(i, \lambda)=(2^{2s-1}, (29\times2^{6s-5}-33\times2^{4s-5}+17\times2^{2s-3}-2)(2^{2s-1}-1)/(2^{2s}-1)).$
\item $(i, \lambda)=(2^{2s-1}-2^{s-1}, \frac{2}{15}\times2^{s-1}(3\times2^{4s}+5\times2^{2s}-8)(2^{2s-1}-2^{s-1}-1)/(2^s+1)).$
\item $(i, \lambda)=(2^{2s-1}+2^{s-1}, \frac{2}{15}\times2^{s-1}(3\times2^{4s}+5\times2^{2s}-8)(2^{2s-1}+2^{s-1}-1)/(2^s-1)).$
\item $(i, \lambda)=(2^{2s-1}-2^s, \frac{7}{3}\times2^{3s-4}(2^{2s-1}-2^s-1)(2^{s-1}-1)).$
\item $(i, \lambda)=(2^{2s-1}+2^s, \frac{7}{3}\times2^{3s-4}(2^{2s-1}+2^s-1)(2^{s-1}+1)).$
\item $(i, \lambda)=(2^{2s-1}-2^{s+1}, \frac{1}{15}\times2^{s-3}(2^{4s-2}-5\times2^{2s-2}+1)(2^{2s-1}-2^{s+1}-1)(2^{s-2}-1)/(2^{2s}-1)).$
\item $(i, \lambda)=(2^{2s-1}+2^{s+1}, \frac{1}{15}\times2^{s-3}(2^{4s-2}-5\times2^{2s-2}+1)(2^{2s-1}+2^{s+1}-1)(2^{s-2}+1)/(2^{2s}-1)).$
\end{itemize}
\end{theorem}

The following Examples \ref{example1} and \ref{example2} from Magma program confirm the main results in Theorems \ref{weight1} and \ref{$2-$design-1}.

\begin{example}\label{example1}
If $s=3,$ then the code ${\overline{{\mathcal{C}_1}^{\bot}}}^{\bot}$ has parameters $[64, 19, 16]$ and weight enumerator
$1+252z^{16}+37632z^{24}+107520z^{28}+233478z^{32}+107520z^{36}+37632z^{40}+252z^{48}+z^{64}.$ It gives $2$-$(64, i, \lambda)$ designs with the
following pairs $(i, \lambda):$
$$(16, 15), (24, 5152), (28, 20160), (32, 57443), (36, 33600), (40, 14560), (48, 141).$$
\end{example}

\begin{example}\label{example2}
If $s=4,$ then the code ${\overline{{\mathcal{C}_1}^{\bot}}}^{\bot}$ has parameters $[256, 25, 96]$ and weight enumerator
$1+17136z^{96}+2437120z^{112}+6754304z^{120}+15137310z^{128}+6754304z^{136}+2437120z^{144}+17136z^{160}+z^{256}.$
\end{example}

It is worth noting that, for $m=4$,
the code ${\overline{{\mathcal{C}_1}^{\bot}}}^{\bot}$ has parameters $[16, 11, 4]$ and weight enumerator
  $1+140z^4+448z^6+870z^8+448z^{10}+140z^{12}+z^{16}$. It forms $3$-$(16, i, \lambda)$ designs with the following pairs $(i, \lambda):$
$$(4, 1), (6, 16), (8, 87), (10, 96), (12, 55).$$

\subsection{Results on the code related to the generalized Kasami case}

We define
\begin{eqnarray}\label{code-2}
  \lefteqn{{\overline{{\mathcal{C}_2}^{\bot}}}^{\bot}:=\{ (\tr_1^s(ax^{2^s+1})+ \tr_1^m(bx^{2^l+1}+cx)+h )_{x \in \mathbb{F}_q}:} \\
&  &\qquad \qquad \qquad \qquad \qquad \quad a \in \mathbb{F}_{2^s}, b, c \in
\mathbb{F}_q, h\in
\mathbb{F}_2 \},\nonumber
\end{eqnarray}
where $\mathcal{C}_2$ is the cyclic code of length $n$  with the parity-check polynomial $M_1(x)M_{2^l+1}(x)M_{2^s+1}(x)$. Note that $\mathcal{C}_2^{\bot}$ is the dual of the extended cyclic code of the parameters satisfying the generalized Kasami case.

For $\overline{{\mathcal{C}_2}^{\bot}}^{\bot}$, we present the main results in the following two theorems.

\begin{theorem}\label{weight2}
Let $1\leq l\leq m-1.$ The weight distributions of the code ${\overline{{\mathcal{C}_2}^{\bot}}}^{\bot}$ over $\mathbb{F}_2$ with length $n+1$ and
$dim({\overline{{\mathcal{C}_2}^{\bot}}}^{\bot})=\frac{5m}{2}+1$ are given in Tables \ref{2} and \ref{3}.
\begin{table}
\begin{center}
\caption{The weight distribution of ${\overline{{\mathcal{C}_2}^{\bot}}}^{\bot}$ when $d'=d$}\label{2}
\begin{tabular}{ll}
\hline\noalign{\smallskip}
Weight  &  Multiplicity   \\
\noalign{\smallskip}
\hline\noalign{\smallskip}
$0$  &  1 \\
$2^{2s-1}-2^{s-1}$  &  $ 2^{2s}(2^s-1)(2^{2(s+d)}-2^{2s+d}-2^{2s}+2^{s+2d}-2^{s+d}+2^{2d})/(2^{2d}-1)$    \\
$2^{2s-1}+2^{s-1}$  &  $ 2^{2s}(2^s-1)(2^{2(s+d)}-2^{2s+d}-2^{2s}+2^{s+2d}-2^{s+d}+2^{2d})/(2^{2d}-1)$    \\
$ 2^{2s-1}-2^{s+d-1}$  &  $ 2^{2(s-d)}(2^{s+d}-1)(2^{2s}-1)/(2^{2d}-1)$\\
$ 2^{2s-1}+2^{s+d-1}$  &  $ 2^{2(s-d)}(2^{s+d}-1)(2^{2s}-1)/(2^{2d}-1)$\\
$ 2^{2s-1}$  &  $ 2(2^{3s-d}-2^{2(s-d)}+1)(2^{2s}-1)$     \\
$ 2^{2s}$  &  $ 1$     \\
\noalign{\smallskip}
\hline
\end{tabular}
\end{center}
\end{table}

\begin{table}
\begin{center}
\caption{The weight distribution of ${\overline{{\mathcal{C}_2}^{\bot}}}^{\bot}$ when $d'=2d$}\label{3}
\begin{tabular}{ll}
\hline\noalign{\smallskip}
Weight  &  Multiplicity   \\
\noalign{\smallskip}
\hline\noalign{\smallskip}
$0$  &  1 \\
$2^{2s-1}-2^{s-1}$  &  $ 2^{2s+3d}(2^s-1)(2^{2s}-2^{2(s-d)}-2^{2s-3d}+2^s-2^{s-d}+1)/(2^{2d}-1)(2^d+1)$    \\
$2^{2s-1}+2^{s-1}$  &  $ 2^{2s+3d}(2^s-1)(2^{2s}-2^{2(s-d)}-2^{2s-3d}+2^s-2^{s-d}+1)/(2^{2d}-1)(2^d+1)$    \\
$ 2^{2s-1}-2^{s+d-1}$  &  $ 2^{2s-d}(2^{2s}-1)(2^s+2^{s-d}+2^{s-2d}+1)/(2^d+1)^2$\\
$ 2^{2s-1}+2^{s+d-1}$  &  $ 2^{2s-d}(2^{2s}-1)(2^s+2^{s-d}+2^{s-2d}+1)/(2^d+1)^2$\\
$ 2^{2s-1}$  &  $ 2(2^{2s}-1)(2^{3s-d}-2^{3s-2d}+2^{3s-3d}-2^{3s-4d}+2^{3s-5d}+2^{2s-d}-2^{2s-2d+1}$\\
&$+2^{2s-3d}-2^{2s-4d}+1)$     \\
$ 2^{2s-1}-2^{s+2d-1}$  &  $ 2^{2s-4d}(2^{s-d}-1)(2^{2s}-1)/(2^d+1)(2^{2d}-1)$\\
$ 2^{2s-1}+2^{s+2d-1}$  &  $ 2^{2s-4d}(2^{s-d}-1)(2^{2s}-1)/(2^d+1)(2^{2d}-1)$\\
$ 2^{2s}$  &  $ 1$     \\
\noalign{\smallskip}
\hline
\end{tabular}
\end{center}
\end{table}
\end{theorem}
Note that the code are six-weight when $d'=d$ and eight-weight when $d'=2d$.

\begin{theorem}\label{$2-$design-2}
Let $s\geq 2$ be a positive integer. Then the supports of the codewords of weight $i$ with ${\overline{{A_i}^{\bot}}}^{\bot}\neq0$ in
${\overline{{\mathcal{C}_2}^{\bot}}}^{\bot}$ give a $2$-design.
Moreover, let $\mathcal{P}=\{0, 1, \ldots, 2^m-1\}$ and $\mathcal{B}$ be the set of the supports of the codewords of
${\overline{{\mathcal{C}_2}^{\bot}}}^{\bot}$ with weight $i,$ where
${\overline{{A_i}^{\bot}}}^{\bot}\neq 0.$ Then ${\overline{{\mathcal{C}_2}^{\bot}}}^{\bot}$ holds $2$-$(2^m, i, \lambda)$ designs for the following
pairs:

(1) if $d'=d,$
\begin{itemize}
\item $(i, \lambda)=(2^{2s-1}-2^{s-1}, 2^{s-1}(2^s-1)(2^{2s-1}-2^{s-1}-1)(2^{2(s+d)}-2^{2s+d}-2^{2s}+2^{s+2d}-2^{s+d}+2^{2d})/(2^{2d}-1)(2^s+1)).$
\item $(i, \lambda)=(2^{2s-1}+2^{s-1}, 2^{s-1}(2^{2s-1}+2^{s-1}-1)(2^{2(s+d)}-2^{2s+d}-2^{2s}+2^{s+2d}-2^{s+d}+2^{2d})/(2^{2d}-1)).$
\item $(i, \lambda)=(2^{2s-1}-2^{s+d-1}, 2^{s-d-1}(2^{s-d}-1)(2^{s+d}-1)(2^{2s-1}-2^{s+d-1}-1)/(2^{2d}-1)).$
\item $(i, \lambda)=(2^{2s-1}+2^{s+d-1}, 2^{s-d-1}(2^{s-d}+1)(2^{s+d}-1)(2^{2s-1}+2^{s+d-1}-1)/(2^{2d}-1)).$
\item $(i, \lambda)=(2^{2s-1}, (2^{2s-1}-1)(2^{3s-d}-2^{2s-2d}+1)).$
\end{itemize}
\end{theorem}

(2) if $d'=2d,$
\begin{itemize}
\item $(i, \lambda)=(2^{2s-1}-2^{s-1},
    2^{3d}(2^{2s}-2^{2(s-d)}-2^{2s-3d}+2^s-2^{s-d}+1)(2^{2s-1}-2^{s-1}-1)(2^{2s-1}-2^{s-1})/(2^{2d}-1)(2^d+1)(2^s+1)).$
\item $(i, \lambda)=(2^{2s-1}+2^{s-1},
    2^{3d}(2^{2s}-2^{2(s-d)}-2^{2s-3d}+2^s-2^{s-d}+1)(2^{2s-1}+2^{s-1}-1)(2^{2s-1}+2^{s-1})/(2^{2d}-1)(2^d+1)(2^s+1)).$
\item $(i, \lambda)=(2^{2s-1}-2^{s+d-1}, (2^s+2^{s-d}+2^{s-2d}+1)(2^{2s-1}-2^{s+d-1})(2^{2s-1}-2^{s+d-1}-1)/2^d(2^d+1)^2).$
\item $(i, \lambda)=(2^{2s-1}+2^{s+d-1}, (2^s+2^{s-d}+2^{s-2d}+1)(2^{2s-1}+2^{s+d-1})(2^{2s-1}+2^{s+d-1}-1)/2^d(2^d+1)^2).$
\item $(i, \lambda)=(2^{2s-1}, 2(2^{2s-1}-1)(2^{3s-d}-2^{3s-2d}+2^{3s-3d}-2^{3s-4d}+2^{3s-5d}+2^{2s-d}-2^{2s-2d+1}+2^{2s-3d}-2^{2s-4d}+1)/2^d).$
\item $(i, \lambda)=(2^{2s-1}-2^{s+2d-1}, (2^{s-d}-1)(2^{2s-1}-2^{s+2d-1})(2^{2s-1}-2^{s+2d-1}-1)/2^{4d}(2^d+1)(2^{2d}-1)).$
\item $(i, \lambda)=(2^{2s-1}+2^{s+2d-1}, (2^{s-d}-1)(2^{2s-1}+2^{s+2d-1})(2^{2s-1}+2^{s+2d-1}-1)/2^{4d}(2^d+1)(2^{2d}-1)).$
\end{itemize}

The following two examples from Magma program confirm the results in Theorems~\ref{weight2} and~\ref{$2-$design-2}.

\begin{example}\label{example4}
If $(s, l)=(2, 1),$ then the code ${\overline{{\mathcal{C}_2}^{\bot}}}^{\bot}$ has parameters $[16, 11, 4]$ and weight enumerator
$1+140z^4+448z^6+870z^8+448z^{10}+140z^{12}+z^{16}.$ It gives $2$-$(16, i, \lambda)$ designs with the
following pairs $(i, \lambda):$
$$(4, 7), (6, 56), (8, 203), (10, 168), (12, 77).$$
\end{example}

\begin{example}\label{example5}
If $(s, l)=(3, 2),$ then the code ${\overline{{\mathcal{C}_2}^{\bot}}}^{\bot}$ has parameters $[64, 16, 24]$ and weight enumerator
$1+5040z^{24}+12544z^{28}+30366z^{32}+12544z^{36}+5040z^{40}+z^{64}.$ It holds $2$-$(64, i, \lambda)$ designs with the
following pairs $(i, \lambda):$
$$(24, 690), (28, 2352), (32, 7471), (36, 3920), (40, 1950).$$
\end{example}

\begin{example}\label{example6}
If $(s, l)=(3, 1),$ then the code ${\overline{{\mathcal{C}_2}^{\bot}}}^{\bot}$ has parameters $[64, 16, 16]$ and weight enumerator
$1+84z^{16}+3360z^{24}+17920z^{28}+22806z^{32}+17920z^{36}+3360z^{40}+84z^{48}+z^{64}.$ It forms $2$-$(64, i, \lambda)$ designs with the
following pairs $(i, \lambda):$
$$(16, 5), (24, 460), (28, 3360), (32, 5611), (36, 5600), (40, 1300), (48, 47).$$
\end{example}

\section{Proofs of the main results}\label{Proofs-main-results}

\subsection{Three lemmas related to the weights of codes}

In order to determine the weight distributions of the two classes of cyclic codes ${\overline{{\mathcal{C}_1}^{\bot}}}^{\bot}$ and ${\overline{{\mathcal{C}_2}^{\bot}}}^{\bot}$, we need the following lemmas.

\begin{lemma}\label{relation}~\cite{Ding182}
Let $\mathcal{C}$ be an $[n, k, d]$ binary linear code, then ${\overline{{\mathcal{C}}^{\bot}}}^{\bot}$ has parameters $[n+1, k+1,
{\overline{d^{\bot}}}^{\bot}]$. Furthermore, ${\overline{{\mathcal{C}}^{\bot}}}^{\bot}$ has only even-weight codewords, and all the nonzero weights in
${\overline{{\mathcal{C}}^{\bot}}}^{\bot}$ are the following:
$$w_1, w_2, \ldots, w_t; n+1-w_1, n+2-w_2, \ldots, n+1-w_t; n+1,$$
where $w_1, w_2, \ldots, w_t$ denote all the nonzero weights of $\mathcal{C}.$
\end{lemma}

The following  Pless power moments given in~\cite{HP03} are  notable variations  of the MacWilliams identities, which is a fundamental result about weight distributions and is a set of linear relations between the weight distributions of  $\mathcal{C}$ and $\mathcal{C}^{\bot}$.

\begin{lemma}\label{power moment identities}
Let $A_i$ and $A_i^{\bot}$ denote the number of code vectors of weight $i$ in a code $\mathcal{C}$ and $\mathcal{C}^{\bot}$, respectively. If
$A_i^{\bot}=0$ for $0 \le i \le 6,$ then the first seven Pless  power moment identities hold:
\begin{eqnarray*}
&&\sum A_i=2^k,\\
&&\sum iA_i=2^{k-1}n,\\
&&\sum i^2A_i=2^{k-2}n(n+1),\\
&&\sum i^3A_i=2^{k-3}(n^3+3n^2),\\
&&\sum i^4A_i=2^{k-4}(n^4+6n^3+3n^2-2n),\\
&&\sum i^5A_i=2^{k-5}(n^5+10n^4+15n^3-10n^2),\\
&&\sum i^6A_i=2^{k-6}(n^6+15n^5+45n^4-15n^3-30n^2+16n),
\end{eqnarray*}
where $k$ denotes the number of information digits.
\end{lemma}

The following lemma given by Luo, Tang and Wang~\cite{LTW10}, gives the weight distributions of the cyclic codes related to the generalized Kasami case.
\begin{lemma}\label{Weight distribution}
The weight distributions of $\mathcal{C}_2$ are given in Tables \ref{4} and \ref{5}.

\begin{table}
\begin{center}
\caption{The weight distribution of ${\mathcal{C}_2}$ when $d'=d$}\label{4}
\begin{tabular}{ll}
\hline\noalign{\smallskip}
Weight  &  Multiplicity   \\
\noalign{\smallskip}
\hline\noalign{\smallskip}
$0$  &  1 \\
$2^{2s-1}-2^{s-1}$  &  $ 2^{s-1}(2^{2s}-1)(2^{2(s+d)}-2^{2s+d}-2^{2s}+2^{s+2d}-2^{s+d}+2^{2d})/(2^{2d}-1)$    \\
$2^{2s-1}+2^{s-1}$  &  $ 2^{s-1}(2^s-1)^2(2^{2(s+d)}-2^{2s+d}-2^{2s}+2^{s+2d}-2^{s+d}+2^{2d})/(2^{2d}-1)$    \\
$ 2^{2s-1}-2^{s+d-1}$  &  $ 2^{s-d-1}(2^{s+d}-1)(2^{2s}-1)(2^{s-d}+1)/(2^{2d}-1)$\\
$ 2^{2s-1}+2^{s+d-1}$  &  $ 2^{s-d-1}(2^{s+d}-1)(2^{2s}-1)(2^{s-d}-1)/(2^{2d}-1)$\\
$ 2^{2s-1}$  &  $ (2^{3s-d}-2^{2(s-d)}+1)(2^{2s}-1)$     \\
\noalign{\smallskip}
\hline
\end{tabular}
\end{center}
\end{table}
\end{lemma}

\begin{table}
\begin{center}
\caption{The weight distribution of ${\mathcal{C}_2}$ when $d'=2d$}\label{5}
\begin{tabular}{ll}
\hline\noalign{\smallskip}
Weight  &  Multiplicity   \\
\noalign{\smallskip}
\hline\noalign{\smallskip}
$0$  &  1 \\
$2^{2s-1}-2^{s-1}$  &  $ \frac{2^{s+3d-1}(2^{2s}-1)(2^{2s}-2^{2(s-d)}-2^{2s-3d}+2^s-2^{s-d}+1)}{(2^{2d}-1)(2^d+1)}$    \\
$2^{2s-1}+2^{s-1}$  &  $ \frac{2^{2s+3d-1}(2^s-1)^2(2^{2s}-2^{2(s-d)}-2^{2s-3d}+2^s-2^{s-d}+1)}{(2^{2d}-1)(2^d+1)}$    \\
$ 2^{2s-1}-2^{s+d-1}$  &  $ 2^{s-1}(2^{2s}-1)(2^s+2^{s-d}+2^{s-2d}+1)(2^{s-d}+1)/(2^d+1)^2$\\
$ 2^{2s-1}+2^{s+d-1}$  &  $ 2^{s-1}(2^{2s}-1)(2^s+2^{s-d}+2^{s-2d}+1)(2^{s-d}-1)/(2^d+1)^2$\\
$ 2^{2s-1}$  &  $ (2^{2s}-1)(2^{3s-d}-2^{3s-2d}+2^{3s-3d}-2^{3s-4d}+2^{3s-5d}$\\
&$+2^{2s-d}-2^{2s-2d+1}+2^{2s-3d}-2^{2s-4d}+1)$     \\
$ 2^{2s-1}-2^{s+2d-1}$  &  $ 2^{s-2d-1}(2^{s-d}-1)(2^{2s}-1)(2^{s-2d}+1)/(2^d+1)(2^{2d}-1)$\\
$ 2^{2s-1}+2^{s+2d-1}$  &  $ 2^{s-2d-1}(2^{s-d}-1)(2^{2s}-1)(2^{s-2d}-1)/(2^d+1)(2^{2d}-1)$\\
\noalign{\smallskip}
\hline
\end{tabular}
\end{center}
\end{table}

\subsection{Quadratic forms}

To determine the parameters of codes ${\overline{{\mathcal{C}_1}^{\bot}}}^{\bot}$ defined in Eq.(\ref{code-1}), we introduce the following function.
\begin{equation}\label{S(a.b.c)}
S(a,b,c)=\sum\limits_{x \in \mathbb{F}_q}(-1)^{\tr_1^m(ax^5+bx^3+cx)},\quad a,b,c \in \mathbb{F}_q.\\
\end{equation}
The first tool to determine the values of exponential sums $S(a,b,c)$ is quadratic forms over $\mathbb{F}_2$.
Let $H$ be an $m\times m$
 matrix over $\mathbb{F}_2$. For the quadratic form
\begin{equation}\label{F(x)}
F: \mathbb{F}^m_2\rightarrow \mathbb{F}_2,\quad F(X)=XHX^T \quad (X=(x_1, x_2, \ldots, x_m)\in \mathbb{F}^m_2),
\end{equation}
we define $r_F$ of $F$ to be the rank of $H+H^T$ over $\mathbb{F}_2$.

The field $\mathbb{F}_q$ is a vector space over $\mathbb{F}_2$ with dimension $m$. We fix a basis $v_1, v_2, \ldots, v_m$ of $\mathbb{F}_q$ over
$\mathbb{F}_2$. Thus each $x\in \mathbb{F}_q$ can be uniquely expressed as
$$x=x_1v_1+x_2v_2+\ldots+x_mv_m \quad (x_i\in \mathbb{F}_2).$$
Then we have the following $\mathbb{F}_2$-linear isomorphism $\mathbb{F}_q \rightarrow\mathbb{F}^m_2:$
$$ \quad x=x_1v_1+x_2v_2+\ldots+x_mv_m\mapsto X=(x_1, \ldots, x_m).$$
With the isomorphism, a function $f: \mathbb{F}_q\rightarrow\mathbb{F}_2$ induces a function $F:\mathbb{F}^m_2\rightarrow\mathbb{F}_2$ where for all
$X=(x_1, \ldots, x_m)\in \mathbb{F}^m_2, F(X)=f(x)$ where $x=x_1v_1+x_2v_2+\ldots+x_mv_m.$
In this way, the function $f(x) = \tr_1^m(wx)$ for $w\in \mathbb{F}_q$ induces a linear form
 $$F(X)=\sum^m_{i=1} \tr_1^m( w v_i)x_i=A_w X^T,$$
where $A_w=(\tr_1^m(w v_1), \ldots, \tr_1^m(w v_m)).$

For $(a, b, c)\in \mathbb{F}^3_q$,  to determine the value of
$$S(a,b,c)=\sum\limits_{x \in \mathbb{F}_q}(-1)^{\tr_1^m(ax^5+bx^3+cx)}=\sum\limits_{X \in \mathbb{F}^m_2}(-1)^{XH_{a,b}X^T+A_cX^T},$$
where $XH_{a,b}X^T$ is the  quadratic form derived from
$f_{a,b}(x) = \tr_1^m (ax^5+bx^3)$ for $a, b\in\mathbb{F}_q$.  We need to determine the rank of $H_{a,b}$ over $\mathbb{F}_2.$ To this end, we have the following
result.

\begin{lemma}\label{rank}
  For $(a, b)\in \mathbb{F}^2_q/{\{(0,0)\}},$ let $r_{a,b}$ be the rank of $H_{a,b}$. Then $r_{a,b}=m$, $m-2$, or $m-4$.
\end{lemma}

\begin{proof}
  It is well known that the rank  of the quadratic form $F(X)$ is defined as the codimension of the $ F_2$ -vector space
$$V = \{y\in  \mathbb{F}_{q}:~ f(x+y)-f(x)-f(y) = 0 \mathrm{~for~ all~} x\in  \mathbb{F}_{q} \}.$$
The cardinality of $V$ is $|V| = 2^{m-r_F}$, where $r_F$ is the rank of $f(x).$

The definition of the function $f_{a,b}(x)$ leads to
$$
f(x+y)-f_{a,b}(x)-f(y) = \tr_1^m ((ax^4+bx^2+a^{2^{m-2}}x^{2^{m-2}}+b^{2^{m-1}}x^{2^{m-1}})y).
$$
Let $$\Phi_{(a,b)}(x)=ax^4+bx^2+a^{2^{m-2}}x^{2^{m-2}}+b^{2^{m-1}}x^{2^{m-1}}.$$
 Then $r_{a,b}=r$ if and only if $\Phi_{(a,b)}(x)=0$ has $2^{m-r_{a,b}}$ solutions in $\mathbb{F}_q.$
On the other hand, since $\Phi_{(a,b)}(x)$ is a $2$-linearized polynomial, then the set of the zeros  to $\Phi_{(a,b)}(x)=0$ is equivalent to that of
$$a^4x^{16}+b^4x^8+b^2x^2+ax=0$$ in
$\mathbb{F}_q$ and forms an $\mathbb{F}_2$-vector space. Since  $r_{a,b}$ is even, $r_{a,b}=m, m-2, m-4.$
We then complete the proof. 
\end{proof}

The following result, which was proved in~\cite{LN97}, will be used in Section~\ref{section-4}.
\begin{lemma}\label{value}~\cite{LN97}
For the fixed quadratic form defined in (\ref{F(x)}), the value distribution of \\
 $\sum\limits_{X \in \mathbb{F}^m_2}(-1)^{F(X)+A_cX^T}$ when $A_c$ runs
through $\mathbb{F}^m_2$, is $0,$ $2^{m-\frac{r_F}{2}},$ or $-2^{m-\frac{r_F}{2}}$.
\end{lemma}

\subsection{Proofs of the main results}\label{section-4}

Now we are ready to give the proofs of our main results.  We begin this subsection by  proving the weight distribution of the code ${\overline{{\mathcal{C}_1}^{\bot}}}^{\bot}$ given in Theorem  \ref{weight1}.

 \begin{proof}[Proof of Theorem~\ref{weight1}]

For each nonzero codeword $\mathbf{c}(a, b, c)=(c_0, \ldots, c_n)$ in $\mathcal{C}_1,$ the Hamming weight of $\mathbf{c}(a, b, c)$ is
\begin{eqnarray}\label{weight formula-1}
w_H(\mathbf{c}(a,b,c))&=&|\{i:0\leq i\leq n-1, c_i\neq 0\}| \nonumber\\
&=&n-|\{i:0\leq i\leq n-1, c_i= 0\}| \nonumber\\
&=&n-\frac{1}{2}\sum^{n-1}_{i=0}\sum^1_{y=0}(-1)^{y\cdot \tr_1^m (a\alpha^{5i}+b\alpha^{3i}+c\alpha^i)}\nonumber\\
&=&n-\frac{n}{2}-\frac{1}{2}\sum_{x\in \mathbb{F}^*_q}(-1)^{\tr_1^m (ax^5+bx^3+cx )} \nonumber\\
&=&\frac{n}{2}+\frac{1}{2}-\frac{1}{2}S(a,b,c) \nonumber\\
&=&2^{2s-1}-\frac{1}{2}S(a,b,c).
\end{eqnarray}
 By Lemmas \ref{rank}-\ref{value} and (\ref{weight formula-1}), we have that the Hamming weight of $\mathbf{c}(a,
b, c)$ is
$$2^{2s-1}, 2^{2s-1}-2^{s-1}, 2^{2s-1}+2^{s-1}, 2^{2s-1}-2^s, 2^{2s-1}+2^s, 2^{2s-1}-2^{s+1}, 2^{2s-1}+2^{s+1}.$$
Plugging these values to the Pless power moments  given by Lemma~\ref{power moment
identities} and after tedious calculations, we obtain
\begin{eqnarray*}
&&A_{2^{2s-1}}=29\times2^{6s-6}-33\times2^{4s-6}+17\times2^{2s-4}-1,\\
&&A_{2^{2s-1}-2^{s-1}}=\frac{1}{15}(3\times2^{6s}+3\times2^{5s}+5\times2^{4s}+5\times2^{3s}-2^{2s+3}-2^{s+3}),\\
&&A_{2^{2s-1}+2^{s-1}}=\frac{1}{15}(3\times2^{6s}-3\times2^{5s}+5\times2^{4s}-5\times2^{3s}-2^{2s+3}+2^{s+3}),\\
&&A_{2^{2s-1}-2^s}=\frac{7}{3}\times2^{3s-4}(2^{3s-1}+2^{2s}-2^{s-1}-1),\\
&&A_{2^{2s-1}+2^s}=\frac{7}{3}\times2^{3s-4}(2^{3s-1}-2^{2s}-2^{s-1}+1),\\
&&A_{2^{2s-1}-2^{s+1}}=\frac{1}{15}\times2^{s-3}(2^{5s-4}+2^{4s-2}-5\times2^{3s-4}-5\times2^{2s-2}+2^{s-2}+1),\\
&&A_{2^{2s-1}+2^{s+1}}=\frac{1}{15}\times2^{s-3}(2^{5s-4}-2^{4s-2}-5\times2^{3s-4}+5\times2^{2s-2}+2^{s-2}-1).
\end{eqnarray*}
The desired conclusion then follows from Lemma \ref{relation}. Thus the proof is completed.
\end{proof}

Then  we prove the affine-invariance  of the code ${\overline{{\mathcal{C}_1}^{\bot}}}^{\bot}$.

\begin{lemma}\label{affine invariant}
  The extended codes $\overline{{\mathcal{C}_1}^{\bot}}$ and $\overline{{\mathcal{C}_2}^{\bot}}$ are affine-invariant.
\end{lemma}

\begin{proof}
We will prove the conclusion with Lemma \ref{Kasami-Lin-Peterson}.
The defining set $T$ of the cyclic code ${\mathcal{C}_1}^{\bot}$ is $T =C_1 \cup C_3 \cup C_5$.  Since $0 \not \in T$, the defining set $\overline{T}$ of
$\overline{{\mathcal{C}_1}^{\bot}}$ is given by $\overline{T} = C_1 \cup C_3 \cup C_5 \cup \{0\}$.
Let $e \in \overline{T} $ and $r \in \mathcal{P}$ . Assume that $e \preceq s$. We need to prove that $r \in \overline{T}$ by Lemma~\ref{Kasami-Lin-Peterson}.

If $r=0,$ then obviously $r\in \overline{T}$. Consider now the case $r>0$. If $e \in  C_1$, then the Hamming weight $wt(e) = 1.$ Since $r \preceq e$,
$wt(r)
= 1.$ Consequently, $r \in C_1\subset  \overline{T}.$  If $e  \in C_3\cup C_5$, then the Hamming weight $wt(e) = 2.$ Since $r \preceq e$, either $wt(r) =
1$
or $r = e.$ In both cases, $r \in  \overline{T}.$ The desired conclusion then follows
from Lemma \ref{Kasami-Lin-Peterson}.

Similarly, we can prove that $\overline{{\mathcal{C}_2}^{\bot}}$  is affine-invariant.

Thus we complete the proof. 
\end{proof}

\begin{proof}[Proof of Theorem~\ref{$2-$design-1}]
From the relation  of ${\overline{{\mathcal{C}_1}^{\bot}}}^{\bot}$ and $\overline{{\mathcal{C}_1}^{\bot}}$, by Lemmas \ref{The dual of an affine-invariant
code} and \ref{affine invariant} , we have  ${\overline{{\mathcal{C}_1}^{\bot}}}^{\bot}$ is  affine-invariant.
Then ${\overline{{\mathcal{C}_1}^{\bot}}}^{\bot}$ holds $2$-designs by  Theorem \ref{2-design}.

Moreover, the number of supports of all codewords with weight $i\neq 0$ in the code ${\overline{{\mathcal{C}_1}^{\bot}}}^{\bot}$ is equal to
${\overline{{A_i}^{\bot}}}^{\bot}$ for each $i,$ where ${\overline{{A_i}^{\bot}}}^{\bot}$ is given in Table \ref{1}. Then the desired conclusions follow
from Eq.(\ref{condition}).
Thus, we finish the proof of Theorems \ref{$2-$design-1}.  
\end{proof}

From all the above, we have finished the proof of the results related to  ${\overline{{\mathcal{C}_1}^{\bot}}}^{\bot}$. Now we prove Theorems
\ref{weight2}
and \ref{$2-$design-2} related to ${\overline{{\mathcal{C}_2}^{\bot}}}^{\bot}$.

\begin{proof}[Proof of Theorem \ref{weight2}]
The desired conclusion follows directly from Lemmas \ref{relation} and \ref{Weight distribution}.
\end{proof}

\begin{proof}[Proof of Theorem~\ref{$2-$design-2}]
The proof is similar to that of Theorem \ref{$2-$design-1}, thus  is omitted here.
\end{proof}

\section{Conclusion}\label{section-5}






In this paper, we determined the weight distributions of two classes of binary cyclic codes. One is  derived from the triple-error correcting BCH code
and the other is from  cyclic codes related to the generalized Kasami case. We proved that both classes of linear codes hold $2$-designs
and explicitly computed their parameters. In particular, we  get  five $3$-designs in ${\overline{{\mathcal{C}_1}^{\bot}}}^{\bot}$ when $m=4.$



\end{document}